\documentclass[12pt]{article}
\usepackage{e-jc}

\usepackage{amsmath,graphicx}

\dateline{Oct 11, 2021}{Nov 30, 2021}{TBD}


\Copyright{The authors. Released under the CC BY-ND license (International 4.0).}

\title{On $1$-subdivisions of transitive tournaments}

\author{Jaehoon Kim\thanks{Supported by the POSCO Science Fellowship of POSCO TJ Park Foundation, and by the KAIX Challenge program of KAIST Advanced Institute for Science-X.}
\and
Hyunwoo Lee\thanks{Supported by the KAIST Undergraduate Research Participation program.}
\and
Jaehyeon Seo\\
\small Department of Mathematical Sciences\\[-0.8ex]
\small KAIST\\[-0.8ex] 
\small Daejeon, South Korea 34141.\\
\small\tt \{jaehoon.kim, hyunwo9216, jaehyeon.seo\}@kaist.ac.kr}

\usepackage[dvipsnames]{xcolor}
\usepackage[a]{esvect} 

\DeclareRobustCommand{\arc}{\vv} 
\DeclareMathOperator{\dist}{dist}

\begin{document}

\maketitle


\begin{abstract}
    The oriented Ramsey number $\vec{r}(H)$ for an acyclic digraph $H$ is the minimum integer $n$ such that any $n$-vertex tournament contains a copy of $H$ as a subgraph.
    We prove that the $1$-subdivision of the $k$-vertex transitive tournament $H_k$ satisfies $\vec{r}(H_k)= O(k^2\log\log k)$. This is tight up to multiplicative $\log\log k$-term.
    
    We also show that if $T$ is an $n$-vertex tournament with $\Delta^+(T)-\delta^+(T)= O(n/k) - k^2$, then $T$ contains a $1$-subdivision of $\vec{K}_k$, a complete $k$-vertex digraph with all possible $k(k-1)$ arcs. This is also tight up to multiplicative constant.
\end{abstract}

\section{Introduction}

What conditions guarantee certain structures on graphs? This is a central problem in extremal graph theory. Tur\'{a}n~\cite{turan1941external} proved that $e(G)> (1-\frac{1}{r-1})\frac{n^2}{2}$ implies that $G$ contains a complete graph $K_r$ on $r$ vertices as a subgraph. 
For a non-bipartite graph $H$, an asymptotic threshold $\bigl(1-\frac{1}{\chi(H)-1}+o(1)\bigr)\binom{n}{2}$ for an $n$-vertex graph $G$ to contain $H$ as a subgraph was determined by Erd\H{o}s-Stone~\cite{erdos1946structure} and Erd\H{o}s-Simonovits~\cite{erdHos1965limit}. This threshold is quadratic in $n$.

For a (di)graph \(H\), subdividing an edge (arc) \(uv\) exactly \(\ell\) times is the operation of replacting the edge (arc) \(uv\) by an (\(\ell+1\))-edge (directed)-path from \(u\) to \(v\) with newly introduced internal vertices. A \emph{subdivision} of \(H\) is a (di)graph obtained by subdividing some edges (arcs) of \(H\). The branch vertices of the subdivision are the vertices which were already there before subdividing edges.
Unlike the above case, if we only want to ensure a subdivision of \(H\) within a graph \(G\), a much weaker bound than quadratic number of edges is sufficient.
In 1967, Mader~\cite{mader1967homomorphieeigenschaften} proved that for given $k$, there exists $f(k)$ such that every graph with average degree $f(k)$ contains $K_k$ as a subdivision.
Mader~\cite{mader1967homomorphieeigenschaften} and Erd\H{o}s-Hajnal~\cite{erdos1969topological} conjectured that this $f(k)$ can be shown to be $O(k^2)$ and this was verified by Bollob\'{a}s and Thomason~\cite{bollobas1998proof} and independently by Koml\'{o}s and Szemer\'{e}di~\cite{komlos1996topological}.

Another key question in extremal combinatorics is a Ramsey-type question. For a given $H$, what values of $n$ ensure that any $2$-coloring on the edges of $K_n$ contains a monochromatic $H$? We write $r(H)$ to denote the smallest such $n$.
In general, such a number $r(H)$ is exponential in $|H|$, as was shown to be
\[
    (\sqrt{2}/e-o(1))k2^{k/2}\le r(K_k)\le e^{-c(\log k)^2}\binom{2k}{k}
\]
where \(c>0\) is an absolute constant. The lower bound is proved by Spencer in \cite{spencer1975ramsey}, and the upper bound is by Sah in \cite{sah2020diagonal}. See \cite{conlon2015recent} for more recent developments in Ramsey theory. However, Alon~\cite{alon1994subdivided} in 1994 proved that if $H$ is a subdivision of another graph obtained by subdividing every edge at least once, the Ramsey number $r(H)$ is linear in $|H|$.
Note that such a graph $H$ is always $2$-degenerate. 
This result was further improved by the celebrated result of Lee~\cite{lee2017ramsey} in 2017 proving the Burr-Erd\H{o}s conjecture stating that any $d$-degenerate graph has a linear Ramsey number. 

There is an analogue considering tournaments instead of complete graphs. A \emph{tournament} is an orientation of a complete graph. For a given oriented graph $H$, we define the \emph{oriented Ramsey number} $\vec{r}(H)$ to be the smallest $n$ where any $n$-vertex tournament contains a copy of $H$.
Indeed, \(\vec{r}(H)\) exists only when \(H\) is acyclic. This is because, no matter how large \(n\) is, the \(n\)-vertex \emph{transitive tournament} \(T_n\), which is the acyclic tournament of order \(n\), does not contain any digraph with a cycle as a subgraph. More generally, for a collection $\mathcal{H}$ of oriented graphs, we define $\vec{r}(\mathcal{H})$ to be the smallest $n$ where any $n$-vertex tournament contains a copy of a graph in $\mathcal{H}$. Again, at least one graph in $\mathcal{H}$ has to be acyclic for the parameter to be defined.

Stearns~\cite{Stearns59} in 1959 and Erd\H{o}s and Moser~\cite{erdos1964problem} in 1964 initiated the study on the oriented Ramsey number and proved that 
\[
    2^{k/2-1}\leq \vec{r}(T_k)\leq 2^{k-1}.
\]
Since then, there has been no improvement to the exponential constants. Note that it is easy to observe $\vec{r}(T_k) \leq r(K_k)$, hence improving the above lower bound is at least as difficult as improving the lower bound on the Ramsey number. 

As in the undirected case, we wonder whether a linear bound on $\vec{r}(H)$ holds if $H$ is obtained by subdividing every arc of another digraph, in particular a subdivision of a transitive tournament. Gir\~{a}o, Popielarz, and Snyder~\cite{girao2021subdivisions} proved that any tournament on $\Omega(k^2)$ vertices contains a subdivision of $T_k$ where each arc is subdivided at most twice (hence each directed path in the subdivision has length at most three). 
In other words, they proved $\vec{r}(\mathcal{H})=O(k^2)$ when $\mathcal{H}$ is the collection of all graphs which are obtained from $T_k$ by subdividing every edge at most twice. 

As subdividing more makes the graph sparser, to consider small number of subdivisions we say a (di)graph is the \emph{\(1\)-subdivision} of \(H\) if it is obtained from \(H\) by subdividing every edge (arc) exactly once.
Let $H_k$ be the $1$-subdivision of $T_k$.
Gir\~{a}o, Popielarz, and Snyder~\cite{girao2021subdivisions} proved that $\vec{r}(H_k) = O(k^2 \log^3 k)$ and posed the following conjecture. 

\begin{conjecture}[Gir\~{a}o, Popielarz, and Snyder]\label{conj:main-conj}
    $\vec{r}(H_k) = O(k^2)$.
\end{conjecture}
This conjecture, if true, is an analogue of Alon's result~\cite{alon1994subdivided} on the subdivision as $H_k$ contains $\binom{k+1}{2}$ vertices, hence $\vec{r}(H_k)$ is linear in $|H_k|$.  In this note, we prove the following theorem improving the result of Gir\~{a}o, Popielarz, and Snyder.

\begin{theorem}\label{thm: main}
	$\vec{r}(H_k) = O(k^2 \log\log k)$.
\end{theorem}

One natural question is whether an analogue of the Burr-Erd\H{o}s conjecture for oriented graphs is true. In other words, this asks if degenerate acyclic oriented graphs have linear (or almost linear) Ramsey number. This was disproved by Fox, He, and Wigderson~\cite{fox2021ramsey}. They showed that for each $\Delta\geq 2$ there exists an acyclic oriented graph $H$ with both maximum outdegree and maximum indegree at most \(\Delta\) such that $\vec{r}(H)=|H|^{\Omega(\Delta^{2/3}/\log^{5/3}(\Delta))}$. They also proved more general bounds on the oriented Ramsey number of digraphs with bounded maximum degree.

As we mentioned earlier, a transitive tournament does not contain any digraph with a cycle, so $\vec{r}(H)$ only makes sense when $H$ is acyclic digraphs.
What if we consider tournaments which are far from being transitive?

Alon, Krivelevich, and Sudakov~\cite{alon2003turan} proved a conjecture of Erd\H{o}s stating that any $n$-vertex graph with at least $\epsilon n^2$ edges contains a $1$-subdivision of a complete graph on $c(\epsilon) \sqrt{n}$ vertices. This states that any graph far from the empty graph contains a subdivision of a large complete graph. In the case of tournaments, we similarly consider an \(n\)-vertex tournament \(T\) that is far from being transitive, and ask whether it contains a $1$-subdivision of a large complete digraph $\vec{K}_k$. Here $\vec{K}_k$ is the digraph on $k$ vertices having \(\arc{uv}\) as an arc for each pair of distinct vertices \(u\) and \(v\), and with no parallel arc.

We prove the following theorem stating that it is indeed possible to find a $1$-subdivision of $\vec{K}_{\Omega(n^{1/3})}$ in a tournament $T$ if we assume that all vertices in $T$ have outdegree at most $n/2+ O(n^{2/3})$.
\begin{theorem}\label{thm: regular subdivision}
	Suppose that $T$ is an $n$-vertex tournament with $\Delta^+(T)-\delta^+(T) \leq \frac{n}{10k} - k^2$.
	Then $T$ contains a $1$-subdivision of $\vec{K}_k$.
\end{theorem}

At first glance assuming all vertices have both out-degree and in-degree $n/2 + O(n/k)$ seems like a very strong assumption. However it turns out that this assumption is necessary. The following proposition shows that the above theorem is tight up to a multiplicative constant and an additive $k^2$ term.

\begin{proposition}\label{prop: example}
For any integers $n$ and $k$ with \(2\le k\le \sqrt{n}\), there exists an $n$-vertex tournament $T$ with $\Delta^+(T) - \delta^+(T) \le 2n/k$ such that $T$ does not contain $1$-subdivision of $\vec{K}_k$.
\end{proposition}
\begin{proof}
Consider an odd number $k' \in \{k-2, k-1\}$ and an $k'$-vertex regular tournament $T'$.
Let $\ell = \lceil n/k' \rceil$.
Blow up each vertex of $T'$ into a transitive tournament of size either $\ell$ or $\ell-1$ to obtain an $n$-vertex tournament $T$. Let $V_1,\dots, V_{k'}$ be the vertex sets of those transitive tournaments.

Assume we have a $1$-subdivision of $\vec{K}_k$ in $T$.
For those $k$ branch vertices, two vertices $u,v$ of them belong to the same part, say $V_i$, of $T$. As those two vertices has exactly same set of out/in-neighbors outside $V_i$, both paths from $u$ to $v$ and $v$ to $u$ must lie inside $V_i$. However, $T[V_i]$ is transitive, so one of two paths does not exist, a contradiction. Hence, $T$ does not contain a $1$-subdivision of $\vec{K}_k$ and
\[
    \Delta^+(T)-\delta^+(T)
    \leq \biggl(\frac{k'-1}{2}\cdot\ell+(\ell-1)\biggr)-\frac{k'-1}{2}\cdot(\ell-1)
    = \frac{k'-1}{2}+\ell-1
    \le \frac{2n}{k},
\]
as desired.	
\end{proof}

\section{Preliminaries}

We write $[n]=\{1,\ldots,n\}$ and $\log = \log_2$. 
For a digraph~$D$, we let~$A(D)$ be the arc set of $D$, and~$d^-_{D}(x)$ and~$d^+_{D}(x)$ refer to the in-degree and out-degree of a vertex $x\in V(D)$, respectively. 
We denote by $\Delta^+(D)$ the maximum out-degree of $D$ and $\delta^+(D)$ the minimum out-degree of $D$.
We denote by~$N_{D}^-(x)$ and~$N_{D}^+(x)$ the in-neighbourhood and out-neighbourhood of $x\in V(D)$, respectively. 

For a (di)graph $G$ and $X\subset V(G)$, we denote by \(G[X]\) the sub(di)graph of $G$ induced by $X$.
	
Let $G$ be an undirected graph and let \(u\), \(v\) be vertices in \(G\). 
We denote by $\dist(u,v)$ the length (which is the same as the number of edges) of a shortest path between $u$ and $v$ in $G$. If such a path does not exist, then we let $\dist(u,v)=\infty$.
For each $i\in\mathbb{N}$, we let $N^i_G(v)$ be the set of all vertices $u$ with $\dist(v,u)=i$ and let $B^i_G(v)= \bigcup_{ 0\leq j\le i} N^i_G(v)$. For \(X\subseteq V(G)\), let \(B^1_G(X)=\bigcup_{v\in X}B^1_G(v)\).

We often omit the subscript $D$ or $G$ when the underlying (di)graph is clear from the context.
We will omit floors and ceilings and treat large numbers as integers whenever it does not affect the argument.

To deal with the $1$-subdivisions of digraphs, it is useful to count the number of directed paths of length two between two vertices.
For a tournament $T$ and $u,v\in V(T)$, let $P_2(u,v)$ be the number of paths of length at most $2$ from $u$ to $v$. As $P_2(u,v) = |N^+(u)\setminus N^+(v)|$, we have
\begin{equation}\label{eq: P_2 deg diff}
    P_2(u,v) - P_2(v,u) = d^+(u) - d^+(v).	
\end{equation}
Observe that $P_2$ satisfies the following triangle inequality.
\begin{proposition}\label{prop: triangular}
For any $u,v,w\in V(T)$, we have
	$P_2(u,v) \leq P_2(u,w)+ P_2(w,v).$
\end{proposition}
\begin{proof}
    We have
    \begin{align*}
        P_2(u,w)+ P_2(w,v)
        &= |N^+(u)\setminus N^+(w)| + |N^+(w)\setminus N^+(v)|
        \\&\geq |N^+(u)| - |N^+(u)\cap N^+(w)\cap N^+(v)| 
        \\&\geq |N^+(u)\setminus N^+(v)|
        = P_2(u,v).
    \end{align*}
\end{proof}

\section{Proofs of the Theorems}
\label{sec:S1Tk}

We first prove Theorem~\ref{thm: main}. Our proof develops the ideas in \cite{girao2021subdivisions}. Let $H_k$ be the $1$-subdivision of $T_k$. We will show that any tournament with at least $Ck^2\log\log k$ vertices contains a copy of $H_k$.
We use induction on $k$ with the choice of $C= 2^{30}$. Note an $m$-vertex tournament contains a transitive tournament $T_{\log m}$ which contains a copy of $H_{\sqrt{\log m}}$. As $\sqrt{\log C} > 5$, we may assume $k \ge 6$. Moreover, as the induction hypothesis, we assume the following.
\begin{equation}\label{eq: IH}
\begin{minipage}{0.90\textwidth}
	for each $s<k$, any tournament on at least $Cs^2\log\log s$ vertices contains a copy of $H_s$.
\end{minipage}	
\end{equation}

Let $n = Ck^2\log\log k$ and $T$ be an $n$-vertex tournament. We enumerate the vertices of $T$ into $v_1,\dots, v_n$ so that 
\begin{equation}\label{eq: descending out-degree}
    d^+(v_1)\geq \dots \geq d^+(v_n).
\end{equation}
We divide $V(T)$ into three sets as follows:
$$
    V_1 = \{v_1, \cdots , v_{n/4}\}, \quad V_2 = \{v_{n/4+1}, \cdots , v_{3n/4}\}, \quad \text{and} \quad V_3 = \{v_{3n/4+1}, \cdots , v_n\}.
$$
By \eqref{eq: IH}, each of $G[V_1]$ and $G[V_3]$ contains a copy of $H_{k/2}$. 

If $d_{T}^+(v_{n/4}) - d_{T}^+(v_{3n/4}) \ge k^2$, then for any vertex $u\in V_1$ and $v\in V_3$, we have $P_2(u, v) \geq k^2$. Hence, there are at least $k^2-1$ paths of length exactly \(2\) from $u$ to $v$. As a copy of $H_k$ contains $\binom{k}{2}+k < k^2-1$ vertices, we can greedily take internally disjoint paths of length $2$ for each pair of vertices between two copies of $H_{k/2}$ until we obtain a copy of $H_{k}$. 
Thus we may assume that, for any $u,v\in V_2$,
\begin{equation}\label{eq: deg diff}
	d_{T}^+(u) - d_{T}^+(v) < k^2.	
\end{equation}

Now we consider an auxiliary graph $G$ recording pairs in $V_2$ with not too many paths of length \(2\) between.
Let 
\begin{equation}\label{eq: G def}
    V(G) = V_ 2 \quad \text{and} \quad E(G) = \{uv \in \binom{V_2}{2} :  P_2(u, v) < k^2  \text{ or } P_2(v, u) < k^2\}.	
\end{equation}
Then by \eqref{eq: P_2 deg diff} and \eqref{eq: deg diff}, for any $uv\in E(G)$ we have 
\begin{equation}\label{eq: P2 small}
	P_2(u,v)<2k^2 \quad\text{and}\quad P_2(v,u) < 2k^2.
\end{equation}
Note that the definition of $G$ ensures that a large independent set in $G$ yields a $1$-subdivision of a large transitive tournament. Moreover, if we have $t$ vertex sets with no edges between them in $G$, we can obtain a large $1$-subdivision as follows.

\begin{claim}\label{cl: many small parts}
	Let $C_1,\dots, C_t \subseteq V(G)$ be disjoint nonempty vertex sets where $G$ has no edges between $C_i$ and $C_j$ for $i\neq j\in [t]$. For each \(i\in[t]\), let \(m_i=\sqrt{|C_i|/(C\log\log k)}\). If \(\sum_{i=1}^t m_i\ge 2k\), then \(T\) contains a copy of \(H_k\).
\end{claim}
\begin{proof}
	Indeed, \eqref{eq: IH} ensures that each $C_i$ contains a copy of \(H_{\lfloor m_i\rfloor}\), and clearly it contains a copy of \(H_1\) which is a vertex. Choose an integer \(0\le n_i\le \max\{1,\lfloor m_i\rfloor\}\) for each \(i\) so that \(\sum_{i=1}^t n_i=k\). If \(t\ge k\), then \(n_i=1\) for \(1\le i\le k\) and \(n_i=0\) for \(i>k\) works; otherwise \(\sum_{i=1}^t (m_i-1)\ge k\), so again such \(n_i\)'s exist. By \eqref{eq: G def}, we can greedily connect different copies of $H_{n_i}$ in \(C_i\) to obtain a copy of $H_{k}$. 	
\end{proof}

The above claim states that if $G$ is not well-connected, then we can obtain a large $1$-subdivision. 
On the other hand, the following claim states that our graph $G$ cannot be too well-connected. The proof of this claim was a part of the argument in \cite{girao2021subdivisions}. We include the proof for completeness.

\begin{claim}\label{claim: bounded expansion}
	For any $v\in V(G)$ and $r\in\mathbb{N}$, we have $|B^r_{G}(v)| \leq 20rk^2$.
\end{claim}
\begin{proof}
    Let $B= B^r_{G}(v)$ and consider the subtournament $T'=T[B]$. Choose a vertex $x$ in $T'$ with out-degree $d^+_{T'}(x) \geq (|T|-1)/2$. Again $T[N^+_{T'}(x)]$ forms another tournament, hence there exists $y\in N^+_{T'}(x)$ having at least $(d^+_{T'}(x)-1)/2$ in-neighbors in $N^+_{T'}(x)$.
    Thus we have
    $$P_2(x,y) \geq |N_{T}^{+}(x)\cap N_{T}^{-}(y)|\geq \frac{1}{4}|B| -1.$$
    However, as \(x,y\in B\), there exists a path $(x=z_1,z_2,\dots, z_s=y)$ of length at most $2r$ between $x$ and $y$ in $G$.
    Hence Proposition~\ref{prop: triangular} together with \eqref{eq: P2 small} implies that 
    $$P_2(x,y)\leq \sum_{i\in [s-1]} P_{2}(z_i,z_{i+1}) \leq 4r k^2.$$
    This yields $|B|\leq 16 rk^2 + 4 \leq 20 rk^2.$
\end{proof}

In order to utilize the above two claims in a right way, we need to quantify the expansion of graphs so that the claims give what we want.

\begin{claim}\label{clm: expansion}
Let $G'$ be a subgraph of $G$. 
For any $v\in V(G')$, there exists $1\leq r\leq \ 1 + \log\log{k}$ such that the following holds where $B^i = B_{G'}^i(v)$:
$$|B^{r}| \leq \frac{1}{10} n^{1/2} |B^{r-1}|^{1/2}.$$
\end{claim}
\begin{proof}
If $|B^1| \leq \frac{1}{10} n^{1/2}$, then $r=1$ suffices, so we assume $|B^1| \geq\frac{1}{10} n^{1/2} > 1.$ 
Suppose that such an $r$ does not exist, meaning that for each $r\in [1 + \log\log{k}]$, we have
\begin{align}\label{eq: expand}
	|B^{r}| > \frac{1}{10} n^{1/2} |B^{r-1}|^{1/2}.
\end{align}
For each $i\in[1+\log\log{k}]$, let $\alpha$ be a real number satisfying $|B^i| = n^{1-\alpha_i}$. 
Then we have $0 \leq \alpha_1 < 1$.

By \eqref{eq: expand}, for each $i\in[1+\log\log{k}]$, we have $n^{1-\alpha_{i}} > \frac{1}{10} n^{1 - \frac{1}{2}\alpha_{i-1}}$, i.e., \(\alpha_i<\frac{1}{2}\alpha_{i-1}+\frac{4}{\log n}\).
Inductively applying this yields
\begin{align*}
    \alpha_{i}
    &< \frac{1}{2} \alpha_{i-1} +  \frac{4}{\log n}
    < \frac{1}{4} \alpha_{i-2} + \frac{4}{\log n} \biggl(1+ \frac{1}{2}\biggr)
    \\&<  \dots
    < 2^{1-i}\alpha_1 + \frac{4}{\log n}\sum_{j=0}^{i-2}2^{-j}
    < 2^{1-i}  + \frac{8}{\log n}.
\end{align*}
Hence, for $s=1+\log\log{k}$, we have 
$\alpha_s <  \frac{1}{\log{k}} + \frac{8}{\log n}$. However, as $k\geq 6$, we have $n= 2^{30}k^2\log\log k < k^{16}$, whence
$\alpha_s <  \frac{24}{\log n}$. Thus
$$|B^s| = n^{1-\alpha_s} \le 2^{\log n - 24} = 2^{-24} n = 2^{-24}C k^2\log\log k.$$
However, Claim~\ref{claim: bounded expansion} implies 
$|B^s|\le |B_G^s(v)| \leq 20 s k^2 \leq 20 k^2 (1+\log\log k) < 2^{-24} C k^2\log\log k$ as \(C=2^{30}\), a contradiction. Hence there exists $r\in[1+\log\log{k}]$ such that $|B^{r+1}| > \frac{1}{10} n^{1/2} |B^r|.$
\end{proof}

Now we are ready to prove Theorem~\ref{thm: main}.
We take distinct vertices $v_1,\dots, v_s$ in $V(G)$, positive integers $r_1,\dots, r_s$, vertex sets $X_1,\dots, X_s$, $Y_1,\dots, Y_s$, and subgraphs $G_1,\dots, G_{s+1}$ of $G$
satisfying the following with the maximum possible $s$.
\begin{itemize}
\item[(G1)] For each $i\in [s]$, we have $r_i\leq 1+\log\log k $.
\item[(G2)] For each $i\in [s]$, we have $G_1=G$ and $G_{i+1} = G - \bigcup_{j=1}^{i} Y_i$ and $X_{i} = B_{G_{i}}^{r_i-1}(v_i)$ and $Y_i = B_{G_{i}}^{r_i}(v_i)$. 
\item[(G3)] For each $i\in [s]$, we have $|Y_i|\leq \frac{1}{10} n^{1/2}|X_i|^{1/2}.$
\end{itemize}
Indeed, such a collection exists as the trivial collection with $s=0$ vacuously satisfies all three conditions.

We claim that with such a choice with the maximum $s$, we have 
$$\sum_{i=1}^{s} |Y_i|  = |V(G)| = \frac{n}{2}.$$
Suppose not. Then $G_{s+1}$ is a nonempty graph, so it contains a vertex. Let $v_{s+1}$ be an arbitrary vertex in $G_{s+1}$. By applying Claim~\ref{clm: expansion} with $G_{s+1}$ and $v_{s+1}$, we can find $r_{s+1} \in [1+ \log\log k]$ such that $B^{r_{s+1}-1}_{G_{s+1}}(v_{s+1}) = X_{s+1}$ and $B^{r_{s+1}}_{G_{s+1}}(v_{s+1}) =Y_{s+1}$ satisfying (G1)--(G3), which contradicts to the maximality of \(s\).

Furthermore, for each $i\in [s]$, as $r_i\leq 1 + \log\log k $, Claim~\ref{claim: bounded expansion} implies that 
\[
    |X_i| \leq |B^{r_i-1}_{G}(v_i)| \leq 20 k^2 \log\log k\leq \frac{n}{100}.
\]
Note that $X_1,\dots, X_s$ are disjoint vertex sets, and for $i< j$, we have $B^1_G(X_i)\cap X_j = Y_i\cap X_j =\emptyset$, so there are no edges between $X_i$ and $X_j$ in $G$. 
In addition, we have
\[
    \sum_{i\in [s]} \sqrt{\frac{|X_i|}{ C\log\log k} } \geq \sum_{i\in [s]} \frac{10 |Y_i|}{n^{1/2} (C\log\log k)^{1/2}} = \frac{5n}{n^{1/2} (C\log\log k)^{1/2}} \geq 2k.
\]
Hence Claim~\ref{cl: many small parts} applies so that $T$ contains a copy of $H_k$. This finishes the proof of Theorem~\ref{thm: main}.

Now we prove Theorem~\ref{thm: regular subdivision}.
Let $\ell = \Delta^+(T)-\delta^+(T)$.
As in the previous section, let $P_2(u,v)$ be the number of paths from $u$ to $v$ of length at most two, and
 we construct an auxiliary graph $G$ with $V(G)=V(T)$ and 
$$E(G)=\{ uv \in\binom{V(G)}{2} : P_2(u,v) < k^2 \enspace \text{or} \enspace P_2(v,u)<k^2\}.$$
As in \eqref{eq: P2 small}, we can use the above to  obtain that for \(uv\in E(G)\),
\begin{align}\label{eq: P2 small 2}
	P_2(u,v)< k^2+\ell \quad\text{and}\quad P_2(v,u)< k^2 +\ell.
\end{align}
Moreover, as in the proof of Claim~\ref{claim: bounded expansion}, we can show that any vertex $v$  satisfies $|B^r_G(v)| \leq 10 r(k^2+\ell)$ for any $r\in \mathbb{N}$.
In particular, with $r=1$, this implies that any vertex $v$ has degree at most $10(k^2+\ell)-1$ in $G$. By Tur\'{a}n's theorem, $G$ contains an independent set of size at least $\frac{n}{10(k^2+\ell)} \geq \frac{n}{10(k^2+ n/(10k)-k^2)} = k$.
Take such an independent set $\{v_1,\dots, v_k\}$ and greedily connect all pairs with internally disjoint paths. Then we obtain a $1$-subdivision of $\vec{K}_k$. This finishes the proof of Theorem~\ref{thm: regular subdivision}.

\subsection*{Acknowledgement}
Shortly after we uploaded our draft to arXiv, Draganic, Correia, Sudakov, and Yuster~\cite{draganic2021ramsey} confirmed Conjecture~\ref{conj:main-conj} by proving a stronger bound \(r(H_k)\leq (2+o(1))k^2\).



\end{document}